\documentclass[12pt]{amsart}
\usepackage{fullpage}

\usepackage[english]{babel}
\usepackage{enumerate}
\usepackage{amsthm}
\usepackage{amsmath}
\usepackage{amssymb}
\usepackage{amsfonts}

\usepackage[nofootnote,draft,silent]{fixme}

\newtheorem{thm}{Theorem}[section]
\newtheorem*{thm*}{Theorem} 
\newtheorem{lem}[thm]{Lemma}
\newtheorem{Def}[thm]{Definition}

\newtheorem{prop}[thm]{Proposition}
\newtheorem{cor}[thm]{Corollary}

\newtheorem{notation}[thm]{Notation}
\newtheorem{crit}[thm]{Criterion}
\theoremstyle{definition}{\newtheorem{ex}[thm]{Example}}

\newcommand{\N}{\ensuremath{\mathbb{N}}}
\newcommand{\C}{\ensuremath{\mathbb{C}}}
\newcommand{\Z}{\ensuremath{\mathbb{Z}}}
\newcommand{\Q}{\ensuremath{\mathbb{Q}}}

\newcommand{\del}{\ensuremath{\partial}}

\newcommand{\G}{\ensuremath{\mathcal{G}}}
\newcommand{\Hcal}{\ensuremath{\mathcal{H}}}
\newcommand{\GL}{\operatorname{GL}}
\newcommand{\SL}{\operatorname{SL}}

\newcommand{\Gm}{\mathbb{G}_m}
\newcommand{\Ga}{\mathbb{G}_a}

\newcommand{\Pcal}{\ensuremath{\mathcal{P}}}

\newcommand{\Gal}{\ensuremath{\underline{\mathrm{Gal}}^{\del \del_t}}}
\newcommand{\GalN}{\ensuremath{\underline{\mathrm{Gal}}^{\del \del_{t_0}}}}
\newcommand{\Galf}{\ensuremath{\mathrm{Gal}}}

\newcommand{\Aut}{\operatorname{Aut}}

\newcommand{\Frac}{\ensuremath{\mathrm{Frac}}}
\newcommand{\h}{\operatorname{-}}

\newcommand{\PPV}{\operatorname{PPV}}

\newcommand{\zu}{\hspace{-0.1cm}>\hspace{-0.15cm}}

\title{New classes of parameterized differential Galois groups}
\author{Annette Bachmayr}
\thanks{The author was funded by the Deutsche Forschungsgemeinschaft (DFG) - grant MA6868/1-1.}

\date{\today}

\begin{document}

\begin{abstract} This paper is on the inverse parameterized differential Galois problem. We show that surprisingly many groups do not occur as parameterized differential Galois groups over $K(x)$ even when $K$ is algebraically closed. We then combine the method of patching over fields with a suitable version of Galois descent to prove that certain groups do occur as parameterized differential Galois groups over $k((t))(x)$. This class includes linear differential algebraic groups that are generated by finitely many unipotent elements and also semisimple connected linear algebraic groups. 
\end{abstract}
\maketitle
\textit{2010 Mathematics Subject Classiﬁcation.} 12H05, 20G15, 14H25, 34M50, 34M03, 34M15.  

\section*{Introduction}
Parameterized differential Galois theory was developed in \cite{CassSin} and it studies the symmetries among solutions of linear differential equations whose coefficients depend on a parameter. More precisely, let $F$ be a field equipped with two commuting derivations $\del$ and $\del_t$ and let $K$ be its field of $\del$-constants, for example $F=\C(t)(x)$ or $F=\C((t))(x)$ with $\del=\del/\del x$, $\del_t=\del / \del_t$ and $K=\C(t)$ or $K=\C((t))$. Note that $K$ is a $\del_t$-differential field. Let $A \in F^{n\times n}$ and consider   the ordinary linear $\del$-differential equation $\del(y)=A\cdot y$ over $F$. A parameterized Picard-Vessiot ring $R$ over $F$ for this equation is an $F$-algebra equipped with extensions of the derivations $\del$ and $\del_t$ such that there exists a fundamental solution matrix $Y \in \GL_n(R)$ (i.e., $\del(Y)=A\cdot Y$) with the property that $R$ is generated by the entries of $Y$, $\det(Y)^{-1}$ and their $\del_t$-derivatives and such that $R$ is a simple $\del$-differential ring with field of $\del$-constants $K$. Parameterized differential Galois theory assigns a parameterized differential Galois group to $R/F$, which can be viewed as a measure of the $\del_t$-algebraic relations among the solutions (i.e., among the entries of $Y$). More precisely, the parameterized Picard-Vessiot group is the group scheme of $\del\del_t$-differential automorphisms of $R/F$. It can be naturally embedded into $\GL_n$ and its image inside $\GL_n$ is defined by $\del_t$-differential algebraic equations over $K$, i.e., it is a linear differential algebraic group over $K$. The inverse problem in parameterized differential Galois theory is the question which linear differential algebraic groups are parameterized Picard-Vessiot groups over $F$. 

If $F=U(x)$ and $\del=\del/\del x$ for some universal differential field $(U,\del_t)$ then a linear differential algebraic group is a parameterized Picard-Vessiot group if and only if it is differentially finitely generated (\cite{Dreyfus}, \cite{MR2975151}). For certain classes of groups, such as linear algebraic groups or linear differential algebraic groups that are unipotent or reductive, there are also complete classifications which groups are differentially finitely generated (\cite{MR2995020}, \cite{MOS}, \cite{MOS2}). 

If $F=K(x)$ with $\del=\del/\del x$ and an arbitrary differential field $(K,\del_t)$ , there is only little known on the inverse parameterized differential Galois problem. Given the known results over $U(x)$, it seems to be natural to believe that every linear differential algebraic group over $K$ that is differentially finitely generated by $K$-rational elements is a parameterized Picard-Vessiot group over $K(x)$. Quite surprisingly, this turns out to be wrong even for subgroups of the multiplicative group $\Gm$ (see Example \ref{ex}.a). Therefore, it seems to be important to study which subgroups of the multiplicative group $\Gm$ and the additive group $\Ga$ are Picard-Vessiot groups over $K(x)$. If $K$ is algebraically closed, we give a classification in Theorem \ref{classification}.

The subsequent parts of the paper treat the case $F=k((t))(x)$ with $\del=\del/ \del x$ and $\del_t=\del/ \del_t$ and $k$ an arbitrary field of characteristic zero. It was shown in \cite{param_LAG} that every connected, semisimple, split linear algebraic group over $k((t))$ is a parameterized Picard-Vessiot group over $F$. The proof relied on a method of patching over fields which was developed by Harbater and Hartmann in \cite{HH}. In this paper, we refine the application of patching to parameterized Picard-Vessiot theory in order to show that linear differential algebraic groups that are generated by finitely many elements with certain properties are parameterized Picard-Vessiot groups over $F$ (Theorem \ref{result}). As a Corollary, we obtain that every linear differential algebraic group over $k((t))$ that is generated by finitely many unipotent $\overline{k((t))}$-rational elements is a parameterized Picard-Vessiot group over $F$ (Theorem \ref{result_unipotent}). In particular, every semisimple (not necessarily split) connected linear algebraic group over $k((t))$ is a parameterized Picard-Vessiot group over $F$ (Corollary \ref{result_semisimple}).

\smallskip

The paper is organized as follows. In Section \ref{sec 1}, we recapitulate parameterized Picard-Vessiot theory and then introduce a Galois descent for parameterized Picard-Vessiot rings (Lemma \ref{lemmainvariantPPVR}). In Section \ref{sec 2}, we classify subgroups of $\Gm$ and $\Ga$ that are parameterized Picard-Vessiot groups over $K(x)$ for $K$ algebraically closed. In Section \ref{section patching}, we recapitulate the method of patching over fields and explain how it can be applied to parameterized Picard-Vessiot theory. As an example, we show how this strategy can be applied to obtain that $\SL_2$ is a parameterized differential Galois group over $k((t))(x)$ (Example \ref{ex SL2}). In Section \ref{section descent}, we combine the method of patching with Galois descent of parameterized Picard-Vessiot rings to obtain a more general statement on patching parameterized Picard-Vessiot rings (Theorem \ref{criterion}). We also translate this theorem into an explicit criterion (Criterion \ref{crit}) that can be applied straight-forward and does not require any knowledge on the method of patching. This criterion states that a given linear differential algebraic group is a parameterized differential Galois group over $k((t))(x)$ if it is generated by finitely many subgroups that are parameterized differential Galois groups over certain overfields. In Section \ref{section results}, we apply this criterion to obtain our results on the inverse parameterized differential Galois problem.   

\smallskip

\textbf{Acknowledgments.} I wish to thank Michael F.\ Singer for helpful discussions concerning subgroups of $\Gm$ that do not occur as parameterized Picard-Vessiot groups over $K(x)$.

\section{Parameterized Picard-Vessiot theory}\label{sec 1}

In this section, we fix some notation and recapitulate parameterized Picard-Vessiot theory. 

All fields are assumed to be of characteristic zero and all rings are assumed to contain $\Q$.
A \textit{$\del\del_t$-ring} $R$ is a ring $R$ with two commuting derivations $\del$ and $\del_t$. Examples of such rings are $\C[t][x]$, $\C[[t]][x]$, $\C[x][[t]]$, $\C(t)(x)$, $\C((t))(x)$, $\C(x)((t))$. A \textit{$\del\del_t$-field} is a $\del\del_t$-ring that is a field. \textit{Homomorphisms of $\del\del_t$-rings} are homomorphisms commuting with the derivations, \textit{$\del\del_t$-ideals} are ideals stable under the derivations and \textit{$\del\del_t$-ring extensions} are ring extensions with compatible $\del\del_t$-structures. Let $(K,\del_t)$ be a differential field. A $\del_t$-$K$-algebra $S$ is a $K$-algebra with an extension $\del_t$ from $K$ to $S$.

Let $R$ be a $\del\del_t$-ring. Then we use the following notation for the corresponding \textit{fields of constants}: $C_R=\{x \in R \mid \del(x)=0\}$ and $R^{\del_t}=\{x \in R \mid \del_t(x)=0\}$. Note that $C_R$ is a $\del_t$-differential ring. A linear differential equation $\del(y)=Ay$ with a matrix $A \in F^{n\times n}$ over a $\del\del_t$-field $F$ is also called a \textit{parameterized (linear) differential equation} to emphasize the extra structure ${\del_t}$ on $F$. A \textit{fundamental solution matrix} for $\del(y)=Ay$ is a matrix $Y \in \GL_n(R)$ for some $\del\del_t$-ring extension $R/F$ such that $\del(Y)=AY$ holds (in other words, the columns of $\del(Y)=AY$ form a fundamental set of solutions of $\del(y)=Ay$). 

\begin{Def}
Let $\del(y)=Ay$ be a parameterized differential equation over a $\del\del_t$-field $F$. A \emph{parameterized Picard-Vessiot extension} for $A$, or \emph{$\PPV$-extension} for short, is a $\del\del_t$-field extension $E$ of $F$ such that 
\begin{enumerate}
 \item There exists a fundamental solution matrix $Y \in \GL_n(E)$ such that ${E=F\<Y\zu_{\del_t}}$ which means that $E$ is generated as a field over $F$ by the entries of $Y$ and all its higher derivatives with respect to ${\del_t}$. 
\item $C_E=C_F$.
\end{enumerate}
\end{Def}
\begin{Def}
A \emph{parameterized Picard-Vessiot ring} for $A$, or \emph{$\PPV$-ring} for short, is a $\del\del_t$-ring extension $R/F$ such that 
\begin{enumerate}
\item There exists a fundamental solution matrix $Y \in \GL_n(R)$ such that ${R=F\{Y,Y^{-1}\}_{\del_t}}$, that is, $R$ is generated as an $F$-algebra by the coordinates of $Y$ and $\det(Y)^{-1}$ and all their higher $\del_t$-derivatives.
\item $C_R=C_F$.
\item $R$ is $\del$-simple, that is, $R$ has no nontrivial $\del$-invariant ideals. 
\end{enumerate}
\end{Def}
A $\PPV$-ring for $A$ always exists if $C_F$ is algebraically closed, see \cite{Wibparam}. Every $\PPV$-extension contains a unique $\PPV$-ring. Indeed, let $E$ be a $\PPV$-extension with fundamental solution matrix $Y \in \GL_n(E)$. Then $R:=F\{Y,Y^{-1}\}_{\del_t}$ is a $\PPV$-ring for $A$. (see \cite{param_LAG} for more details.) 

\begin{Def}
Let $R$ be a $\PPV$-ring over a $\del\del_t$-field $F$ and denote $K=C_F$. Then the \emph{parameterized differential Galois group of $R/F$}, or \emph{$\PPV$-group} for short, is the group functor 
\[ \underline{\Aut}^{\del\del_t}\colon \underline{{\del_t}\h K\h\text{algebras}}\to \underline{\text{Groups}}, \ S \mapsto \operatorname{Aut}^{\del\del_t}(R\otimes_K S / F \otimes _K S), \] where $\operatorname{Aut}^{\del\del_t}(R\otimes_K S / F \otimes _K S)$ denotes the set of $(F\otimes_K S)$-algebra automorphisms of $R\otimes_K S$ that commute with both $\del$ and $\del_t$ (where we extend the derivation $\del$ from $R$ to $R\otimes_K S$ via $\del|_S=0$).
\end{Def}

Let $(K,\del_t)$ be a differential field. A \emph{linear differential algebraic group} or \emph{linear $\del_t$-algebraic group} over $K$ is a group functor $\G \colon  \underline{{\del_t}\h K\h\text{algebras}}\to \underline{\text{Groups}}$ such that there exists an $n \in \N$ and a system $$\{p_\alpha \mid \alpha \in I \}\subseteq K[\del_t^{k}(X_{ij}) \ | \ k \in \N_{\geq 0}, 1 \leq i,j \leq n]$$ of differential polynomials in $n^2$ variables with coefficients in $K$ such that for all $\del_t$-$K$-algebras $S$: ${\G(S)=\{g \in \GL_n(S) \ | \ p_\alpha(g)=0 \text{ for all } \alpha \in I \}}$. If $\G$ is a linear $\del_t$-algebraic group over $K$ and $L/K$ is an extension of $\del_t$-fields, the base change of $\G$ from $K$ to $L$ is denoted by $\G_L$. If $\G_1,\dots,\G_r$ are differential algebraic subgroups of $\G$, then $\overline{\left \langle \G_1,\dots,\G_r \right \rangle}^K$ is defined as the smallest linear differential subgroup $\Hcal \subseteq \G$ containing $\G_1,\dots,\G_r$ (the superscript ``K'' refers to Kolchin closure which is the differential algebraic analog of the Zariski closure). Similarly, if $g \in \G(K)$, then $\overline{\left \langle g \right \rangle}^K$ is defined as the smallest linear differential algebraic subgroup $\Hcal$ of $\G$ such that $g \in \Hcal(K)$. 

\begin{ex}\label{ex UG Ga}
 The additive group $\Ga$ is a linear $\del_t$-algebraic group. We usually work with $\Ga(S)=(S,+)$ although technically speaking we would have to consider it as a subgroup of $\GL_2$, e.g., $\Ga(S)=\{ \begin{pmatrix}
1 & x \\ 0 &1
\end{pmatrix} \mid x \in S\}$. In contrast to the fact that $\Ga$ does not have non-trivial algebraic subgroups, it does have a lot of differential algebraic subgroups: $\Ga^L$ with $\Ga^L(S)=\{ x\in S \mid L(x)=0\}$ is a differential algebraic subgroup of $\Ga$ for every linear differential operator $L \in K[\del_t]$. It is easy to see that $\Ga^L$ is a subgroup of $\Ga^{\tilde L}$ if and only if $L$ is a right-divisor of $\tilde L$. 
If $(K,\del_t)$ is a universal differential field, then every differential algebraic subgroup of $\Ga$ is of this form by \cite[Prop. 11]{Cassidy}.
\end{ex}

\begin{ex}\label{ex Kolchin Abschluss}
 If $g \in \Ga(K)$ with $g \neq 0$, then $$\overline{\left \langle g \right \rangle}^K=\{x \in \Ga \mid g\del_t(x)-\del_t(g)x=0 \}$$ and this group has no non-trivial differential algebraic subgroups. For a proof, set $H=\{x \in \Ga \mid g\del_t(x)-\del_t(g)x=0 \}=\{ x \in \Ga \mid \del_t(x/g)=0 \}$. Note that $H$ is a differential algebraic subgroup of $\Ga$ and $g \in H(K)$. Thus $\overline{\left \langle g \right \rangle}^K \leq H$. On the other hand, there exists a linear differential operator $L \in K[\del_t]$ with $\overline{\left \langle g \right \rangle}^K=\Ga^L$ where we use the notation $\Ga^L$ as in Example \ref{ex UG Ga}. As $H=\Ga^{\tilde L}$ with $\tilde L=g\del_t-\del_t(g)\del_t^0$, we conclude that $L$ divides $\tilde L$ from the right. However, $\tilde L$ is of order one and does not have non-trivial right divisors. Thus $\Ga^{\tilde L}=\Ga^L$ and the claim follows.  
\end{ex}

\begin{thm}\label{thmgalois}
Let $F$ be a $\del\del_t$-field with field of $\del$-constants $K$ and let $A \in F^{n\times n}$. Assume that there exists a $\PPV$-ring $R$ for the parameterized differential equation $\del(y)=Ay$. Then the $\PPV$-group of $R/F$ becomes a linear ${\del_t}$-algebraic group over $K$ via the following natural embedding into $\GL_n$ depending on a fixed fundamental solution matrix $Y \in \GL_n(R)$:
\[\theta_S\colon\operatorname{Aut}^{\del\del_t}(R\otimes_K S / F \otimes _K S)\hookrightarrow \GL_n(S), \ \sigma \mapsto (Y )^{-1}\sigma(Y ).\] 
\end{thm}
\begin{proof}
See \cite[Thm. 1.4]{param_LAG}
\end{proof}

The image of $\theta_S$ is denoted by $\Gal_Y(R/S) \leq \GL_n$ and will also be called the $\PPV$-group of $R/F$. If $\G\leq \GL_n$ is a given linear differential algebraic group and $F$ is a $\del\del_t$-field, we say that \textit{$\G$ is a $\PPV$-group over $F$} if there exists a linear differential equation $\del(y)=Ay$ over $F$ such that there exists a $\PPV$-ring $R/F$ for $\del(y)=Ay$ with $\Gal_Y(R/F)=\G$ for a suitable fundamental solution matrix $Y \in \GL_n(R)$. The following proposition asserts that this does not depend on the fixed representation $\G \hookrightarrow \GL_n$.

\begin{prop}\label{Tannaka}
Let $F$ be $\del\del_t$-differential field and set $K=C_F$. Let $R/F$ be a $\PPV$-ring with $\PPV$-group $G$. Let $\G\leq \GL_n$ be a faithful $K$-representation of $G$ such that for some fundamental solution matrix $Y \in \GL_n(R)$, we have $\Gal_Y(R/F)=\G$. Let $\tilde \G\leq \GL_m$ be another representation of $G$. Then there exists a $\PPV$-ring $\tilde R \subseteq R$ together with a fundamental solution matrix $\tilde Y \in \GL_m(R)$ such that $\Gal_{\tilde Y}(\tilde R/F)=\tilde \G$. Moreover, if $\tilde \G$ is a faithful representation, then $\tilde R =R$. 
\end{prop}
\begin{proof}
Set $A=\del(Y)Y^{-1} \in F^{n\times n}$ and let $M$ be the $\del$-differential module $(F^n,\del_M)$, where $\del_M$ is given by $A$. Then $\Frac(R)$ is a $\PPV$-extension for $M$ (see \cite[Def. 3.27]{Gilletetc} for a definition). Let further $\underline{\operatorname{DMod}}(F,\del)$ denote the category of $\del$-differential modules over $(F,\del)$. It is shown in \cite[Thm 5.1]{Gilletetc} that there is a canonical $\del_t$-differential structure on $\underline{\operatorname{DMod}}(F,\del)$. Let $\underline{\operatorname{Rep}}(G)$ denote the category of finite-dimensional differential algebraic representations of $G$ over $K$. Using Theorem 5.5 together with Theorem 4.27 in \cite{Gilletetc}, there is an equivalence of differential categories \[ \left \langle M \right \rangle _{\otimes, \del_t}\to \underline{\operatorname{Rep}}(G) ,\] where $\left \langle M \right \rangle _{\otimes, \del_t}$ denotes the minimal full
rigid $\del_t$-subcategory of $\underline{\operatorname{DMod}}(F,\del)$ that contains $M$ and is closed under taking subquotients (compare also with the proofs \cite[Lemma 8.2, Prop. 8.6]{Gilletetc}). Any element in $\left \langle M \right \rangle _{\otimes, \del_t}$ has a fundamental solution matrix with entries inside $R$ and the claim follows then similarly as in the non-parameterized case (see for example \cite[Prop. 3.2]{HHM}).
\end{proof}

We conclude this section with a lemma that provides us with a Galois descent for $\PPV$-rings. 
 \begin{lem}\label{lemmainvariantPPVR}
 Let $K/K_0$ be a finite Galois extension of $\del_t$-differential fields with (finite) Galois group $\Gamma$. Let $F_0$ be a $\del\del_t$-differential field with $C_{F_0}=K_0$ and let $F$ be the $\Gamma$-Galois field extension $F=F_0\otimes_{K_0} K$ of $F_0$. Note that $F$ is a $\del\del_t$-field extension of $F_0$ (in a unique way). Let further $L/F$ be an extension of $\del\del_t$-differential fields with $C_{L}=C_F=K$ and such that the action of $\Gamma$ on $F$ over $F_0$ extends to an action on $L$ as $\del\del_t$-differential automorphisms. If $R=F\{Y,Y^{-1}\}_{\del_t}\subseteq L$ is a $\PPV$-ring over $F$ such that $Y \in \GL_n(L)$ is invariant under the action of $\Gamma$, then $R_0:=F_0\{Y,Y^{-1}\}_{\del_t}$ is a $\PPV$-ring over $F_0$ with $\Gal_{Y}(R_0/F_0)_K=\Gal_Y(R/F)$ inside $\GL_n$. In particular, if $\Gal_Y(R/F)=\G_K$ for a linear $\del_t$-algebraic group $\G \leq \GL_n$ defined over $K_0$, then $\Gal_{Y}(R_0/F_0)=\G$.
 \end{lem}
\begin{proof}
 First, note that the entries of $A:=\del(Y)Y^{-1}$ are $\Gamma$-invariant, hence $A \in F_0^{n\times n}$. Since $\Gamma$ acts as $\del\del_t$-automorphisms, $L^\Gamma$ is a $\del\del_t$-field  extension of $F_0$ and $R_0\subseteq L^\Gamma$. Also, $C_{L^\Gamma}=C_L^\Gamma=K^\Gamma=K_0=C_{F_0}$, hence $R_0\subseteq L^\Gamma$ is a $\PPV$-ring for $A$ over $F_0$. As $\Frac(R_0)$ and $K$ are linearly disjoint over $K_0$, we obtain an isomorphism of $\del\del_t$-rings $R_0\otimes_{K_0} K \to R$. Hence $R_0\otimes_{K_0}S$ is canonically isomorphic to $R\otimes_K S$ for every $\del_t$-$K$-algebra $S$ and thus $\Aut^{\del\del_t}(R_0\otimes_{K_0} S/F_0\otimes_{K_0} S)$ and $\Aut^{\del\del_t}(R\otimes_K S/F\otimes_K S)$ are canonically isomorphic. Therefore, we obtain an equality $\Gal_{Y}(R_0/F_0)_K(S)=\Gal_Y(R/F)(S)$ inside $\GL_n(S)$ for every $\del_t$-$K$-algebra $S$ and the claim follows.
\end{proof}

\section{Subgroups of $\Gm$ and $\Ga$}\label{sec 2}
If $(K,\del_t)$ is a differential field and $L\in K[\del_t]$ is a linear differential equation, we denote
\[\Gm^{L\circ \bigtriangleup}=\{\lambda \in \Gm \ | \  L\left(\frac{\del_t(\lambda)}{\lambda}\right)=0 \} .\]
Note that $\Gm^{L \circ \bigtriangleup}$ is a linear $\del_t$-algebraic subgroup of the multiplicative group $\Gm=\GL_1$ and that $\Gm^{L \circ \bigtriangleup}$ is a subgroup of $\Gm^{\tilde L \circ \bigtriangleup}$ if and only if $L$ is a right divisor of $\tilde L$. If $K$ is a universal differential field, then Corollary 2 to Proposition 31 in \cite{Cassidy} implies that every linear $\del_t$-algebraic subgroup of $\Gm$ is either finite or of the form $\Gm^{L\circ \bigtriangleup}$ for some $L \in K[\del_t]$. 

Recall that we also defined linear $\del_t$-algebraic subgroups $\Ga^L$ of the additive group $\Ga$ in Example \ref{ex UG Ga}.

\begin{prop}\label{order one}
Let $F$ be a $\del\del_t$-field, $K=C_F$ and $L \in K[\del_t]$ a linear differential equation. 
\begin{enumerate}
 \item If $\Gm^{L\circ \bigtriangleup}$ is a $\PPV$-group over $F$, then there exists an $a \in F^\times$ such that $\Gm^{L\circ \bigtriangleup}$ is the $\PPV$-group of a $\PPV$-ring for the differential equation $\del(y)=ay$ of order one over $F$.
\item If $\Ga^L$ is a $\PPV$-group over $F$, then there exist a $\PPV$-ring over $F$ with $\PPV$-group $\Ga^L$ of the following form: $R=F\{y\}_{\del_t}$ for some $y \in R$ with $\del(y) \in F$. 
\end{enumerate}
\end{prop}
\begin{proof}
Part a) is a direct consequence of Proposition \ref{Tannaka}. 

To prove Part b), let $\G$ be the representation of $\Ga^L$ with 
\[\G(S)=\{\begin{pmatrix}
           1 & \mu \\ 0 & 1 
          \end{pmatrix} \mid \mu \in S \text{ with } L(\mu)=0\}   \]
for all $\del_t$-$K$-algebras $S$. Then by Proposition \ref{Tannaka}, there exists a $\PPV$-ring $R/F$ with fundamental solution matrix $Y \in \GL_2(R)$ such that $\Gal_Y(R/F)=\G$. As $\Gal_{BY}(R/F)=\Gal_Y(R/F)$ for all $B \in \GL_2(F)$, we may assume that the entries $y_{11}$ and $y_{21}$ of $Y$ are both non-zero. For all $\del_t$-$K$-algebras $S$ and all $\sigma \in \Aut^{\del \del_t}(R\otimes_K S/F\otimes_K S)$, there exists a $\mu_\sigma \in S$ with $L(\mu_\sigma)=0$ and 
\[ Y^{-1}\sigma(Y)=\begin{pmatrix}
                    1 & \mu_\sigma \\ 0 & 1
                   \end{pmatrix}
 .\]  Hence
\begin{eqnarray}
  \sigma(y_{11})=y_{11} \text{ and } \sigma(y_{21})=y_{21} \label{eq func inv} \\
\sigma(y_{12})=y_{12}+\mu_\sigma \cdot y_{11} \text{ and } \sigma(y_{22})=y_{22}+\mu_\sigma \cdot y_{21} \label{eq mu}
\end{eqnarray}
for all $\sigma \in \Aut^{\del \del_t}(R\otimes_K S/F\otimes_K S)$. Equation (\ref{eq func inv}) implies that $y_{11}$ and $y_{21}$ are functorially invariant under the action of $\underline{\Aut}^{\del \del_t}(R/F)$, so $y_{11}$ and $y_{21}$ are contained in $F$ (see \cite[Prop. 8.5]{Gilletetc}). After replacing $Y$ with $B\cdot Y$ for the diagonal matrix $B \in \GL_2(F)$ with entries $y_{11}^{-1}$ and $y_{21}^{-1}$ we may assume that $Y$ is of the form 
\[Y=\begin{pmatrix}
     1 & y \\ 1 & z
    \end{pmatrix}\] 
for some elements $y,z \in R$ with $y \neq z$. Note that Equation (\ref{eq mu}) implies that $y-z$ is also functorially invariant and thus contained in $F$. Hence 
\[B'=\begin{pmatrix}
      1 & 0 \\ (y-z)^{-1} & -(y-z)^{-1}
     \end{pmatrix}\] is contained in $\GL_2(F)$ and after replacing $Y$ with $B'Y$, we may assume that 
\[Y=\begin{pmatrix}
     1 & y \\ 0 & 1 
    \end{pmatrix}.\]Then $\del(y) \in F$ since $\del(Y)Y^{-1}$ has entries in $F$, and $R=F\{y\}_{\del_t}$ since $R$ is generated as a $\del_t$-$F$-algebra by the entries of $Y$ and $Y^{-1}$.
\end{proof}
Alternatively, Proposition \ref{order one} can also be proven by using Kolchin's differential cohomology theory (\cite[Chapter VII]{Kolchin}). A classification of $\Ga^L$-torsors is given in \cite[VII.6, Cor. 1]{Kolchin} and the $\Gm^{L\circ \bigtriangleup}$-torsors can be classified using Prop. 8 in \cite[VII.6]{Kolchin} applied to the exact sequence $1 \to \Gm^{L\circ \bigtriangleup} \to \Gm  \xrightarrow{L\circ \bigtriangleup} \Ga \to 1$. We claim that all entries of

\begin{lem} \label{AequivalenzUG}
Let $F$ be a $\del\del_t$-field with $C_F=K$ and let $L \in K[\del_t]$.
\begin{enumerate}
 \item  Let $R=F\{y,y^{-1}\}_{\del_t}$ be a $\PPV$-ring of a differential equation $\del(y)=ay$ of order one over $F$ and set $\G=\Gal_y(R/F) \leq \GL_1$. Then $\G$ is contained in $\Gm^{L\circ \bigtriangleup}$ if and only if $L\left(\frac{\del_t(y)}{y}\right) \in F$.
\item Let $R=F\{y\}_{\del_t}$ be a $\PPV$-ring over $F$ with $\del(y) \in F$ and set $\G=\Gal_Y
(R/F) \leq \Ga$ (in its two-dimensional representation), where $Y=\begin{pmatrix}
                                                                                                1 & y \\ 0 & 1
                                                                                               \end{pmatrix}$. Then $\G$ is contained in $\Ga^{L}$ (in its two-dimensional representation) if and only if $L(y) \in F$.  
\end{enumerate}

\end{lem}
\begin{proof}
a) The element $L\left(\frac{\del_t(y)}{y}\right)\in R$ is containd in $F$ if and only if it is functorially invariant under the action of $\underline{\Aut}^{\del\del_t}(R/F)$ (by \cite[Prop. 8.5]{Gilletetc}). Now $L\left(\frac{\del_t(y)}{y}\right)$ is functorially invariant if and only if $\sigma\left( L\left(\frac{\del_t(y)}{y}\right) \right)=L\left(\frac{\del_t(y)}{y}\right) $ for every $\del_t$-$K$-algebra $S$ and every $\sigma \in \Aut^{\del\del_t}(R\otimes_K S/F\otimes_K S)$. For such a $\sigma$, we denote $\lambda_\sigma=y^{-1}\sigma(y\otimes1) \in \G(S)\subseteq S^\times$. It is now straight-forward to compute that
\[\sigma\left( L\left(\frac{\del_t(y)}{y}\right) \right)=L\left(\frac{\del_t(y)}{y}\right)  +L\left(\frac{\del_t(\lambda_\sigma)}{\lambda_\sigma}\right).\] Therefore, $L\left(\frac{\del_t(y)}{y}\right)$ is functorially invariant if and only if $L\left(\frac{\del_t(\lambda_\sigma)}{\lambda_\sigma}\right)=0$ for all $\del_t$-$K$-algebras $S$ and all $\sigma \in \Aut^{\del\del_t}(R\otimes_K S/F\otimes_K S)$. Since $\G(S)=\{\lambda_\sigma \mid \sigma \in \Aut^{\del\del_t}(R\otimes_K S/F\otimes_K S) \}$, this is equivalent to $\G(S)\subseteq \Gm^{L\circ \bigtriangleup}(S)$ for all $\del_t$-$K$-algebras $S$ which is equivalent to $\G \leq\Gm^{L\circ \bigtriangleup}$.

b) Again, $L(y)$ is contained in $F$ if and only if it is functorially invariant under the action of $\underline{\Aut}^{\del\del_t}(R/F)$. For every $\del_t$-$K$-algebra $S$ and every $\sigma \in \Aut^{\del\del_t}(R\otimes_K S/F\otimes_K S)$, there exists a $\mu_\sigma \in S$ with $\sigma(y )=y  +  \mu_{\sigma}$. Hence $\sigma(L(y ))=L(\sigma(y ))=L(y )+L(\mu_\sigma)$. Therefore, $L(y)$ is functorially invariant if and only if $L(\mu_\sigma)=0$ for all $\sigma$ which is equivalent to $\{\begin{pmatrix}                                                                                                                                                                                                                                                                                                                                                                                                                                                                                                                               1& \mu_\sigma \\ 0 & 1                                                                                                                                                                                                                                                                                                                                                                                                                                                                                                                               \end{pmatrix}
 \mid \sigma \in \Aut^{\del\del_t}(R\otimes_K S/F\otimes_K S) \}\subseteq\Ga^L(S)$ for all $S$ and the claim follows.
\end{proof}

\begin{lem}\label{fundamentalset}
Let $(K,\del_t)$ be a differential field and let $L \in K[\del_t]$ be a linear differential equation of order $n$. 
\begin{enumerate}
 \item There exists a fundamental set of solutions of $L \circ \del_t$ inside $K$ if the following holds: For every linear differential equation $\tilde L \in K[\del_t]$ of order less than $n$, there exists an element $b \in K$ with $L(\del_t(b))=0$ and $\tilde L(\del_t(b))\neq0$.
 \item There exists a fundamental set of solutions of $L$ inside $K$ if the following holds: For every linear differential equation $\tilde L$ in $K[\del_t]$ of order less than $n$, there exists an element $b \in K$ with $L(b)=0$ and $\tilde L(b)\neq0$.
\end{enumerate}

\end{lem}
\begin{proof}
a) Choose a $K^{\del_t}$-basis $b_1,\dots,b_m$ of the solution space of $L \circ \del_t$ inside $K$. We may assume $b_1=1$. We claim that $m=n+1$ under the assumption of a).  Let $\ell \in K[\del_t]$ be the Wronskian $\ell(y)=W(b_1,\dots,b_m,y)$. This is a linear differential equation of order $m\leq n+1$ with solution space spanned by $1,b_2,\dots, b_m$. As $\ell(1)=0$, $\ell=\tilde L \circ \del_t$ for some $\tilde L \in K[\del_t]$ of order $m-1\leq n$. The solution spaces of $L\circ \del_t$ and $\tilde L\circ \del_t=\ell$ inside $K$ thus coincide. By assumptions, this implies that the order of $\tilde L$ equals $n$. We obtain $m-1=n$ and the claim follows. 

The proof of b) is similar.
\end{proof}

Let $K$ be an algebraically closed field (of characteristic zero) with a derivation $\del_t$. We consider the $\del\del_t$-field $K(x)$, where the derivation $\del$ is defined as $\del=\del/\del x$.

We use partial fraction decomposition for elements in $K(x)$. Recall that for $g \in K(x)$, there exist unique elements $g_0 \in K[x]$, $\beta_1,\dots,\beta_r \in K$ and $\gamma_1,\dots,\gamma_r,\gamma_{21},\dots,\gamma_{sr} \in K$ such that 
\[g(x)=g_0(x)+\sum\limits_{i=1}^r\frac{\gamma_i}{x-\beta_i}+\sum\limits_{i=1}^r\frac{\gamma_{2i}}{(x-\beta_i)^2}+\dots+\sum\limits_{i=1}^r\frac{\gamma_{si}}{(x-\beta_i)^s}.\] The term $\sum\limits_{i=1}^r\frac{\gamma_{i}}{x-\beta_i}$ is called the \textit{logarithmic part} of $g$ and $g$ has a $\del$-antiderivative inside $K(x)$ if and only if its logarithmic part is zero.  

\begin{prop}\label{necessary condition}
Let $K$ be an algebraically closed field with a derivation $\del_t$ and let $L \in K[\del_t]$ be a linear differential equation. 
\begin{enumerate}
 \item If $\Gm^{L\circ \bigtriangleup}$ is a $\PPV$-group over $K(x)$, then there exists a fundamental set of solutions of $L\circ\del_t$ inside $K$.  
\item  If $\Ga^L$ is a $\PPV$-group over $K(x)$, then there exists a fundamental set of solutions of $L$ inside $K$. 
\end{enumerate}

\end{prop}
\begin{proof}Let $n$ be the order of $L$.

a) By Proposition \ref{order one}.a, there exists a $\PPV$-ring $R=F\{y,y^{-1}\}_{\del_t}$ for a differential equation $\del(y)=ay$ of order one over $K(x)$ such that $\Gal_y(R/F)=\Gm^{L\circ \bigtriangleup}\leq \GL_1$. We set $f=\frac{\del_t(y)}{y} \in R$. For any linear differential equation $\ell \in K[\del_t]$, we compute
\begin{equation*}\label{equation L}
\ell(\del_t(a))=\ell\left(\del_t\left(\frac{\del(y)}{y} \right)\right)=\ell\left(\del\left(\frac{\del_t(y)}{y} \right)\right)=\ell(\del(f))=\del(\ell(f)).
\end{equation*}
Together with Lemma \ref{AequivalenzUG}.a, this implies that there is a $\del$-antiderivative for $L(\del_t(a))$ inside $K(x)$ and that there is no $\del$-antiderivative for $\tilde L(\del_t(a))$ inside $K(x)$ for any linear differential equation $\tilde L \in K[\del_t]$ of order less than $n$ (otherwise, $\Gm^{L\circ \bigtriangleup}\leq \Gm^{\tilde L\circ \bigtriangleup}$, a contradiction).  

Let $\sum_{i=1}^r\frac{\gamma_{i}}{x-\beta_i}$ be the logarithmic part of $a$ (with $\beta_i$, $\gamma_i \in K$). Then the logarithmic part of $\del_t(a)$ equals $\sum_{i=1}^r\frac{\del_t(\gamma_{i})}{x-\beta_i}$. For $\ell \in K[\del_t]$, the logarithmic part of $\ell(\del_t(a))$ equals $\sum_{i=1}^r\frac{\ell(\del_t(\gamma_{i}))}{x-\beta_i}$. Therefore, there exists a $\del$-antiderivative for $\ell(\del_t(a))$ inside $K(x)$ if and only if $\ell(\del_t(\gamma_{i}))=0$ for all $1\leq i \leq r$. We conclude that $L(\del_t(\gamma_{i}))=0$ for all $1\leq i \leq r$ and that for any linear differential equation $\tilde L \in K[\del_t]$ of order less than $n$ there exists an $i \leq r$ with $\tilde L(\del_t(\gamma_{i}))\neq0$. The claim now follows from Lemma \ref{fundamentalset}.a.

b) By Proposition \ref{order one}.b, there exists a $\PPV$-ring $R=F\{y\}_{\del_t}$ with $\del(y) \in K(x)$ and $\Gal_Y(R/F)=\Ga^L\leq \GL_2$, where $Y=\begin{pmatrix} 1 & y \\ 0 & 1 \end{pmatrix}$. Set $a=\del(y)\in K(x)$ and let $\sum_{i=1}^r \frac{\gamma_i}{x-\beta_i}$ be its logarithmic part. Then $\ell(a)$ has logarithmic part $\sum_{i=1}^r \frac{\ell(\gamma_i)}{x-\beta_i}$ for all $\ell \in K[\del_t]$. As $\ell(a)=\del(\ell(y))$, Lemma \ref{AequivalenzUG}.b implies that there is a $\del$-antiderivative for $L(a)$ in $K(x)$ and that there is no $\del$-antiderivative for $\tilde L(y)$ in $K(x)$ for all $\tilde L \in K[\del_t]$ of order less than $n$. We conclude that $L(\gamma_i)=0$ for all $i=1,\dots,r$ and that for all $\tilde L \in K[\del_t]$ of order less than $n$, there exists a $j \leq r$ with $\tilde L(\gamma_j)\neq 0$. The claim now follows from Lemma \ref{fundamentalset}.b.
\end{proof}

\begin{ex}\label{ex}
Let $K=\overline{k((t))}$ be an algebraic closure of a Laurent series field in characteristic zero equipped with the usual $\del_t$-derivation. 
\begin{enumerate}
 \item Consider $L=t\del_t+\del_t^0 \in K[\del_t]$. The solution space of $L\circ \del_t$ inside $K$ is $\overline{k}$, hence there exists no fundamental solution set of $L\circ \del_t$ inside $K$ (this is due to the fact that $K$ does not contain a logarithm of $t$). Proposition \ref{necessary condition}.a now implies that the group 
\[\G=\Gm^{L\circ \bigtriangleup}=\{\lambda \in \Gm \mid \del_t\left(t\frac{\del_t(\lambda)}{\lambda}\right)=0 \} \] does \underline{not} occur as a $\PPV$-group over $K(x)$. In particular, $\G$ does not occur as a $\PPV$-group over $k((t))(x)$, even though $\G$ differentially generated by one $k((t))$-rational element: $\G=\overline{\<t\zu}^K $.
\item Consider $L=\del_t^2+\frac{1}{t}\del_t \in K[\del_t]$. The solution space of $L$ inside $K$ is $\overline{k}$, hence there exists no fundamental solution set of $L$ inside $K$. Proposition \ref{necessary condition}.b implies that the group $\Ga^L$ is not a $\PPV$-group over $K(x)$. However, this is not as surprising as part a) of this example, since $\Ga^L$ is not differentially generated by $K$-rational elements. 
\end{enumerate}
\end{ex}

We proceed by showing that the converse of Proposition \ref{necessary condition} also holds (even if $K$ is not algebraically closed).
\begin{prop}\label{sufficient condition}
Let $(K,\del_t)$ be a differential field and let $L$ be a non-zero linear differential equation in $K[\del_t]$.
\begin{enumerate}
 \item If there exists a fundamental set of solutions of $L\circ\del_t$ inside $K$, then $\Gm^{L\circ \bigtriangleup}$ is a $\PPV$-group over $K(x)$.
\item If there exists a fundamental set of solutions of $L$ inside $K$, then $\Ga^L$ is a $\PPV$-group over $K(x)$.
\end{enumerate} 
\end{prop}
\begin{proof}
a) We may assume that $L$ is of degree $n\geq 1$ (otherwise, adjoining $e^x$ yields a $\PPV$-ring with $\PPV$-group the constant group $\Gm^\bigtriangleup$). 
Set $m=n+1$ and choose a $K^{\del_t}$-basis $b_1,\dots,b_m \in K$ of the solution space of $L\circ \del_t$. We define
\[a=\frac{b_1}{x-1}+\frac{b_2}{x-2}+\dots+\frac{b_m}{x-m} \in K(x) \] 
and consider the differential equation $\del(y)=ay$ of order one over $K(x)$. Note that $a \in K[[x]]^{\times}$, hence there exists a $y \in K[[x]]^\times$ with $\del(y)=ay$. Then $K(x)\<y \zu_{\del_t}\subseteq K((x))$ is a $\PPV$-extension for $\del(y)=ay$ (here we use $C_{K((x))}=K$) and thus $R=K(x)\{y,y^{-1}\}_{\del_t}$ is a $\PPV$-ring for $\del(y)=ay$ over $K(x)$. We define $\G=\Gal_y(R/K(x))$ and claim that $\G=\Gm^{L\circ \bigtriangleup}$.

Set $f=\frac{\del_t(y)}{y} \in R$. Inside $K((x))$, we can write $f$ as
\[f=\log(x-1)\del_t(b_1)+\log(x-2)\del_t(b_2)+\dots+\log(x-m)\del_t(b_m)+c \]
for some $c \in K$. Thus
\[L(f)=\sum\limits_{i=1}^m \log(x-i)\cdot L(\del_t (b_i))+L(c)=L(c) \in K \] and Lemma \ref{AequivalenzUG}.a implies $\G \leq \Gm^{L\circ \bigtriangleup}$. 

We assumed that $L$ is of degree $n \geq 1$, hence $m \geq 2$ and thus $\del_t(b_i)\neq 0$ for some $i$. It follows that $f$ is transcendental over $K(x)$ and thus $y$ is transcendental over $K(x)$. 
Therefore, the (non-parameterized) differential Galois group of the (non-parameterized) Picard-Vessiot-ring $K(x)[y,y^{-1}]$ for $\del(y)=ay$ equals $\Gm$. By \cite[Prop. 3.6.(2)]{CassSin}, this implies that $\G$ is Zariski-dense inside $\Gm$.

 By \cite[Cor. 2 to Prop. 31]{Cassidy}, any Zariski-dense, Kolchin-closed subgroup of $\Gm$ is of the form $\Gm^{\tilde L\circ \bigtriangleup} $ for some linear differential equation $\tilde L \in \tilde K[\del_t]$ and some $\del_t$-differential field extension $\tilde K/K$. Let $\tilde L$ be such an equation with $\G=\Gm^{\tilde L\circ \bigtriangleup}$. Lemma \ref{AequivalenzUG}.a implies that $\tilde L(f)=\sum_{i=1}^m \log(x-i)\cdot\tilde L\circ \del_t (b_i)+\tilde L(c)\in \tilde K(x)$ and thus $\tilde L \circ \del_t(b_i)=0$ for all $1 \leq i \leq m$. We conclude that the fundamental solution space of $L\circ \del_t$ is contained in the solution space of $\tilde L \circ \del_t$. Therefore, there exists a linear differential equation $Q \in \tilde K[\del_t]$ with $\tilde L\circ \del_t=Q\circ L \circ \del_t$, so $\tilde L=Q\circ L$ which implies $\Gm^{L\circ \bigtriangleup} \leq \Gm^{\tilde L\circ \bigtriangleup}=\G$ and the claim follows.  

b) Again, we may assume that $L$ is of degree $n \geq 1$ (otherwise, adjoining $\log(x)$ yields a $\PPV$-ring with $\PPV$-group the constant group $\Ga^{\del_t^0}$). By assumptions, there exist a fundamental solution set $b_1,\dots,b_n$ of $L$ in $K$. Define 
\[a=\frac{b_1}{x-1}+\dots+\frac{b_n}{x-n} \in K(x).\] Note that $a \in K[[x]]$, hence there exists a unique element $y \in K[[x]]$ with $\del(y)=a$ and with constant term $1$. Using Lemma \ref{AequivalenzUG}.b, it is easy to check that $R=F\{y\}_{\del_t} \subseteq K((x))$ is a $\PPV$-ring over $K(x)$ with $\PPV$-group $\Ga^L$.
\end{proof}

Recall that in Example \ref{ex}.a, we found a subgroup of $\Gm$ that is $\del_t$-differentially finitely generated by $K$-rational elements, but $G$ is \textit{not} a $\PPV$-group over $K(x)$. 
The following corollary shows that this phenomenon does not occur in the additive case. 

\begin{cor}
Let $(K,\del_t)$ be a differential field and let $G$ be a subgroup of $\Ga$ that is $\del_t$-differentially finitely generated by $K$-rational elements. Then $G$ is a $\PPV$-group over $K(x)$.
\end{cor}
\begin{proof}
Assume that $G$ is $\del_t$-differentially generated by $a_1,\dots,a_r \in K$. Let $L \in K[\del_t]$ be the Wronskian $\operatorname{W}(a_1,\dots,a_r,y)$. Then $G=\Ga^L$ and the claim follows from Proposition \ref{sufficient condition}.b.  
\end{proof}

We combine Proposition \ref{necessary condition} and Proposition \ref{sufficient condition} and obtain a classification of those subgroups of $\Gm$ that are $\PPV$-groups over $K(x)$.

\begin{thm}\label{classification}
Let $(K,\del_t)$ be an algebraically closed differential field. 
Then for every non-zero linear differential equation $L \in K[\del_t]$, $\Gm^{L\circ \bigtriangleup}$ is a $\PPV$-group over $K(x)$ if and only if there exists a fundamental set of solutions of $L\circ\del_t$ inside $K$ and $\Ga^{L}$ is a $\PPV$-group over $K(x)$ if and only if there exists a fundamental set of solutions of $L$ inside $K$.  
\end{thm}

Note that the $\Gm$ and $\Ga$ itself are not $\PPV$-groups over $K(x)$ for any differential field $(K,\del_t)$, since otherwise their base change from $K$ to a universal closure $U$ of $K$ would be a $\PPV$-group over $U(x)$ which cannot be true by \cite{Dreyfus}, since $\Ga$ and $\Gm$ are not differentially finitely generated. 

\section{Patching and parameterized differential Galois theory} \label{section patching}
Patching over fields is a method which was established by Harbater and Hartmann in \cite{HH}. We briefly recall all definitions needed for our purpose.
We fix a field $k$ of characteristic zero, and consider $F=k((t))(x)$ as a $\del\del_t$-differential field via $\del=\del/\del x$ and $\del_t=\del/\del t$. For the Galois descent in the next section we need a change of variables $z=\frac{x}{t}$. Note that $F=k((t))(z)$ with $\del(z)=1/t$ and $\del_t(z)=-z/t$. In particular, $\del_t(z)\neq 0$. We fix pairwise distinct elements $q_1,\dots,q_m \in k$ (which will be specified in Section \ref{section descent}) and consider the closed points $P_1,\dots,P_m$ on the $z$-line $\mathbb{P}^1_k$ defined by $z=q_1,\dots,z=q_m$. We define $2m+1$ overfields $F_U$, $F_{P_1},\dots,F_{P_m}$, $F_{\wp(P_1)}$,\dots,$F_{\wp(P_m)}$ of $F$ that will be used throughout the rest of the paper. 
\begin{eqnarray*}
F&=& k((t))(x)=k((t))(z) \\
F_U&=&\operatorname{Frac}(k[(z-q_1)^{-1},\dots,(z-q_m)^{-1}][[t]])\\
F_{P_i}&=&k((t,z-q_i)):=\operatorname{Frac}(k[[t]][[z-q_i]]) \\
F_{\wp(P_i)}&=&k((z-q_i))((t)).
\end{eqnarray*}
Note that $k[[t]][[z-q_i]]=k[[z-q_i]][[t]]$, hence $F \subseteq F_{P_i} \subseteq F_{\wp(P_i)}$ for each $1 \leq i \leq m$. Also, $F\subseteq F_U$ and $F_U\subseteq k(z)((t))\subseteq F_{\wp(P_i)}$ for each $i$, hence we have a diagram of fields $F \subseteq F_U, F_{P_i} \subseteq F_{\wp(P_i)}$ for each $1\leq i \leq m$. Note that $\del\colon F\to F$ extends canonically to derivations on all these fields compatibly with the inclusions $F \subseteq F_U, F_{P_i} \subseteq F_{\wp(P_i)}$. On $F_{\wp(P_i)}$, this extension is given by 
\[\del \colon k((z-q_i))((t))\to k((z-q_i))((t)), \ \sum \limits_{n=r}^\infty f_n t^n \mapsto \sum \limits_{n=r}^\infty \frac{\del f_n}{\del(z-q_i)} t^{n-1}.  \] 
Moreover, $C_F=k((t))=C_{F_{\wp(P_i)}}$ for all $1 \leq i \leq m$ and in particular $C_{F_{P_i}}=k((t))=C_F$ for $i=1,\dots,m$. In the next section, we will slightly modify the derivation $\del_t \colon F \to F$ to a certain derivation $\del_{t_0}$ and then also define extensions of $\del_{t_0} \colon F \to F$ to all of these fields (see Lemma \ref{defdelN}). In the next theorem, we thus allow arbitrary commuting derivations $\del$ and $\del_t$ on $F$.  

\begin{thm}\label{thmpatching} 
Let $n \in \N$ and set $\Pcal=\{P_1,\dots,P_m\}$. Let $\del$ and $\del_t$ be commuting derivations on $F$. Assume that $\del$ and $\del_t$ extend to derivations on the fields $F_U$, $F_{P_i}$ and $F_{\wp(P_i)}$ such that $\del$ and $\del_t$ commute, such that they are compatible with all inclusions $F_U \subseteq F_{\wp(P_i)}$ and $F_{P_i} \subseteq F_{\wp(P_i)}$ and such that $C_{F_{\wp(P_i)}}=k((t))$ for all $i$. 
For all $P \in \Pcal$, let $R_P/F_P$ be a $\PPV$-ring for a linear differential equation $\del(y)=A_Py$ with $A_P \in F_P^{n\times n}$ such that $R_P \subseteq F_{\wp(P)}$ and let $\G_P \leq \GL_n$ be the $\PPV$-group $\Gal_{Y_P}(R_P/F_{P})$ for a fixed fundamental solution matrix $Y_P \in \GL_n(R_P)$. Then there exists a $\PPV$-ring $R/F$ with fundamental solution matrix $Y \in \GL_n(R)$ such that the $\PPV$-group $\G=\Gal_Y(R/F)\leq \GL_n$ is the Kolchin closure of the group generated by all $\G_P$: $\G=\overline{\left\langle \G_{P_1},\dotsm \G_{P_m} \right\rangle}^K\leq \GL_n$. Moreover, $R\subseteq F_U$.
\end{thm}
\begin{proof}
 This is a result from \cite{param_LAG}. More precisely, this was proven in \cite[Thm 2.2]{param_LAG} for fixed derivations $\del,\del_t$ on the fields $F, F_{P_i}, F_{\wp(P_i)}, F_U$ (with $\del(z)=1$, $\del(t)=0$; and $\del_t(z)=0, \del_t(t)=1$). However, all that was needed in the proof was that the derivations $\del,\del_t$ on these fields commute and are compatible with the inclusions $F \subseteq F_U, F_{P_i} \subseteq F_{\wp(P_i)}$ and that their fields of $\del$-constants equal $k((t))$. Therefore, the result transfers to the slightly more general setup of Theorem \ref{thmpatching}. 
\end{proof}

Theorem \ref{thmpatching} is designed for applications on the inverse parameterized differential Galois problem over $k((t))(x)$. The following example gives a sample application. 
\begin{ex}\label{ex SL2}
In this example, we would like to show that $\SL_2$ is a PPV-group over $F=k((t))(x)=k((t))(z)$. In a first step, we construct subgroups $\G_1,\dots,\G_4$ of $\SL_2$ that generate a Kolchin-dense subgroup of $\SL_2$. We do this in a way that the subgroups $\G_1,\dots,\G_4$ are of a very simple shape, hoping that we can then show that these subgroups are PPV-groups over $F_{P_1},\dots,F_{P_4}$ for suitable points $P_1,\dots,P_4$ on the $z$-line. Consider 
$$\G_1=\{ \begin{pmatrix}
          1 & \alpha \\
0 & 1 
         \end{pmatrix} \mid \del_t(\alpha)=0 \}, \ \G_2=\{ \begin{pmatrix}
          1 & 0 \\
\alpha & 1 
         \end{pmatrix} \mid \del_t(\alpha)=0 \},$$ 
$$\G_3=\{ \begin{pmatrix}
          1 & \alpha \\
0 & 1 
         \end{pmatrix} \mid \del_t(\alpha/t)=0 \}, \ \G_4=\{ \begin{pmatrix}
          1 & 0 \\
\alpha & 1 
         \end{pmatrix} \mid \del_t(\alpha t)=0 \}. $$
Note that $\G_1,\dots,\G_4$ are $\del_t$-algebraic subgroups of $\SL_2$ defined over $k((t))$. It can be shown that $\overline{<\G_1,\G_2,\G_3,\G_4>}^K=\SL_2$ (see for example Proposition 3.1. in \cite{param_LAG}). Fix pairwise distinct elements $q_1,\dots,q_4 \in k$. Let $P_1,\dots,P_4$ be the closed points on the $z$-line defined by $z=q_1,\dots,z=q_4$ and consider the nine overfields $F_{P_1}$, $F_{P_2}$, $F_{P_3}$, $F_{P_4}$, $F_{\wp(P_1)}$, $F_{\wp(P_2)}$, $F_{\wp(P_3)}$, $F_{\wp(P_4)}$, $F_U$ of $F$ as defined above. By Theorem \ref{thmpatching} it suffices to construct matrices $A_1,\dots, A_4$ with $A_i \in F_{P_i}^{2\times 2}$ such that there exist matrices $Y_1,\dots,Y_4$ with $Y_i \in \GL_2(F_{\wp(P_i)})$ and $\del(Y_i)=A_iY_i$ and furthermore such that the $\PPV$-ring $R_{P_i}=F_{P_i}\{Y_i,Y_i^{-1}\}_{\del_t}$ has $\PPV$-group $\G_i$ over $F_{P_i}$. The initial problem of finding a $\PPV$-ring over $F$ with $\PPV$-group $\SL_2$ is thus reduced to finding four $\PPV$-rings over the overfields $F_{P_1},\dots,F_{P_4}$ with $\PPV$-groups $\G_1,\dots,\G_4$. As $\G_1,\dots,\G_4$ are isomorphic to subgroups of $\mathbb{G}_a$, the latter task can be solved by choosing suitable logarithmic differential equations. We refer to the proof of Thereom 4.2. in \cite{param_LAG} for explicit differential equations.
 
\end{ex}

\section{A criterion on the inverse problem}\label{section descent}
\begin{notation}\label{not}
 In this section, $k_0$ is a field (of characteristic zero), $K_0=k_0((t_0))$ is a Laurent series field and we set $\del_{t_0}=\del/\del t_0$ on $K_0$. We define $F_0=K_0(x)$ and consider it as a $\del\del_{t_0}$-differential field with $\del=\del/\del x$ and $\del_{t_0}=\del/ \del t_0$. 
We consider the finite extension $K=k((t))$ of $K_0$, where $k/k_0$ is a finite Galois extension and $t$ denotes an $e$-th root of $t_0$ ($e\geq 1$), such that $k$ contains a primitive $e$-th root of unity $\zeta$. We consider $K$ as a $\del_{t_0}$-differential field extension of $K_0$ in the unique way (i.e., $\del_{t_0}(t)=t^{1-e}/e$). Similarly, we consider the field $F=K(x) \cong F_0\otimes_{K_0} K$ as a $\del\del_{t_0}$-field extension of $F_0$ in the unique way. We define $\Gamma=\Galf(K/K_0)\cong \Galf(F/F_0)$.
\end{notation}

First note that if $A \in F_0^{n\times n}$ such that there exists a $\PPV$-ring $R_0$ for the linear differential equation $\del(y)=Ay$ over $F_0=K_0(x)$ with $\GalN_Y(R_0/F_0)=\G \leq \GL_n$, then it is immediate from the definitions that $R_0\otimes_{K_0} K$ is a $\PPV$-ring for $\del(y)=Ay$ over $F_0\otimes_{K_0}K\cong K(x)$ with $\PPV$-group $\GalN_{Y }(R_0\otimes_{K_0} K/K(x))=\G_{K}$. For the inverse problem over $F_0=K_0(x)$, it is often easier to construct $\PPV$-rings over $K(x)$ for a suitable finite extension $K/K_0$ (depending on the group we would like to realize) instead of over $F_0$. In this section, we give a criterion that ensures that a $\PPV$-ring over $K(x)$ constructed using Theorem \ref{thmpatching} is of the form $R_0\otimes_{K_0} K$ for a $\PPV$-ring $R_0$ over $F_0$. For the non-parameterized case, this has been worked out in \cite{HHM}. The arguments in the parameterized case go along the same lines. 

As in Section \ref{section patching}, we define $z=\frac{x}{t}$. Note that $\del_{t_0}(z)=-zt^{-e}/e$. For every $\sigma \in \Gamma$, there exists a unique $n_\sigma$ with $0\leq n_\sigma \leq e-1$ such that $\sigma(t)=\zeta^{n_\sigma}t$. 
Details of the following facts can be found in \cite[Ex. 2.3]{HHM}. The action of $\Gamma$ on $F$ induces an action on the $z$-line $\mathbb{P}^1_k$ that sends a finite $k$-point $P$ of the form $z=q$ to the point $z=\zeta^{n_{\sigma^{-1}}}\cdot\sigma^{-1}(q)$. For each $\sigma \in \Gamma$, there are extensions of $\sigma \colon F\to F$ to $\del$-differential isomorphisms  
\[\sigma \colon F_{P^\sigma}=k((z-\zeta^{n_{\sigma^{-1}}}\sigma^{-1}(q),t)) \to F_{P}=k((z-q,t))\] and 
\[ \sigma \colon F_{\wp(P^\sigma)}=k((z-\zeta^{n_{\sigma^{-1}}}\sigma^{-1}(q)))((t)) \to F_{\wp(P)}=k((z-q))((t)).\] (They map $z-\zeta^{n_{\sigma^{-1}}}\sigma^{-1}(q)$ to $\zeta^{-n_\sigma}(z-q)$.) The isomorphism $\sigma \colon F_{\wp(P^\sigma)} \to F_{\wp(P)}$ restricts to the isomorphism $\sigma \colon F_{P^\sigma} \to F_{P}$. 

 It was shown in \cite[Lemma 4.4]{HHM} that for any $r \in \N$, there exist $r$ closed, finite $k$-points $P_1,\dots,P_r$ on the $z$-line $\mathbb{P}^1_k$ such that the orbits $P_1^\Gamma,\dots,P_r^\Gamma$ are disjoint and each of order $|\Gamma|$. We consider these $m:=r\cdot |\Gamma|$ points $P$ and the corresponding fields $F_P, F_{\wp(P)}, F_U$ as explained in Section \ref{section patching}. (We remark that unless $e=1$ the variable change $x\mapsto z$ is necessary in order that the orbits consists of $|\Gamma|$ points.) As we chose our set of $m$ points $\Gamma$-invariant, the action of $\Gamma$ on $F$ extends to an action of $\Gamma$ on $F_U$ as $\del$-differential automorphisms, compatible with all inclusions $F_U \subseteq F_{\wp(P)}, F_{\wp(P^\sigma)}$. In other words, $\sigma \colon F_U \to F_U$ is the restriction of $\sigma \colon F_{\wp(P^\sigma)} \to F_{\wp(P)}$ for every $P$.

\begin{lem} \label{defdelN}
The derivation $\del_{t_0} \colon k((t))(z) \to k((t))(z)$ extends to derivations $\del_{t_0} \colon F_U \to F_U$, $\del_{t_0} \colon F_{P} \to F_P$ and $\del_{t_0} \colon F_{\wp(P)} \to F_{\wp(P)}$ for all $P$. These derivations are compatible with the inclusions $F_U \subseteq F_{\wp(P)}$ and $F_P \subseteq F_{\wp(P)}$ and they commute with $\del$. Moreover, for every $\sigma \in \Gamma$, the $\del$-differential-isomorphisms $\sigma \colon F_{P^\sigma} \to F_P$, $\sigma \colon F_{\wp(P^\sigma)} \to F_{\wp(P)}$ and $\sigma \colon F_U \to F_U$ as defined above are actually $\del\del_{t_0}$-differential isomorphisms.
\end{lem}
\begin{proof}
 We first define derivations $\del_{t_0} \colon F_{\wp(P)} \to F_{\wp(P)}$ for all $P$ and then show that they restrict to derivations $\del_{t_0} \colon F_{P} \to F_P$ and to one single derivation $\del_{t_0} \colon F_U \to F_U$. Let $P$ be a point in one of the orbits  $P_1^\Gamma,\dots,P_r^\Gamma$ and let $q \in k$ such that $P$ is defined by $z=q$. Then $F_{\wp(P)}=k((z-q))((t))$ and we define 
\[ \del_{t_0} \colon k((z-q))((t)) \to k((z-q))((t)), \ \sum \limits_{n=r}^\infty f_nt^n \to \sum \limits_{n=r}^\infty \left(\frac{n}{e}f_n -  \frac{z}{e}\frac{\del f_n}{\del (z-q)}\right)t^{n-e}. \]
It is easy to check that this map is a derivation that commutes with $\del\colon F_{\wp(P)} \to F_{\wp(P)}$ and that it restricts to $\frac{\del}{\del t_0}$ on $F=k((t))(z)$. If $f \in k[[z-q]][[t]]$, then $\del_{t_0}(f) \in t^{-e}k[[z-q]][[t]]$. Thus $\del_{t_0}$ maps elements of $F_P$ (which is defined as the fraction field of $k[[z-q]][[t]]$) to elements of $F_P$ and it thus restricts to a derivation $\del_{t_0} \colon F_P \to F_P$. Let now $q_1,\dots,q_m \in k$ be such that the unit of the orbits  $P_1^\Gamma,\dots,P_r^\Gamma$ consists precisely of the points defined by $z=q_1,\dots,z=q_m$. Recall that $F_U=\Frac(k[(z-q_1)^{-1},\dots,k(z-q_m)^{-1}][[t]])$ and note that $\del/ \del(z-q_i) \colon k((z-q_i))\to k((z-q_i))$ restricts to 
\[\del/\del z \colon k[(z-q_1)^{-1},\dots,k(z-q_m)^{-1}] \to k[(z-q_1)^{-1},\dots,k(z-q_m)^{-1}]\] for every $i=1,\dots,m$. Thus if $f \in k[(z-q_1)^{-1},\dots,k(z-q_m)^{-1}][[t]]$ then \[ \del_{t_0}(f) \in t^{-e}k[(z-q_1)^{-1},\dots,k(z-q_m)^{-1}][[t]][z] \] and we conclude that $\del_{t_0} \colon k((z-q_i))((t)) \to k((z-q_i))((t))$ restricts to a derivation $\del_{t_0} \colon F_U \to F_U$ for every $i=1,\dots,m$ and that this restriction does not depend on $i$. 
 
It remains to show that for all $\sigma \in \Gamma$, $\sigma \colon F_{\wp(P^\sigma)} \to F_{\wp(P)}$ is a $\del_{t_0}$-isomorphism (with respect to the derivations $\del_{t_0} \colon F_{\wp(P^\sigma)} \to F_{\wp(P^\sigma)}$ and $\del_{t_0} \colon F_{\wp(P)} \to F_{\wp(P)}$) for every $P \in P_1^\Gamma\cup\dots\cup P_r^\Gamma$. Let $P$ be such a point and let $q \in k$ be such that $P$ is defined by $z=q$. Then $P^\sigma$ is the point defined by $z=\zeta^{n_{\sigma^{-1}}}\cdot\sigma^{-1}(q)$ and $\sigma \colon F_{\wp(P^\sigma)} \to F_{\wp(P)}$ can be written explicitly as 
\[ \sigma \colon k((z-\zeta^{n_{\sigma^{-1}}}\sigma^{-1}(q) ))((t)) \to k((z-q))((t)) \]
\[ \sum \limits_{i=r}^\infty \sum \limits_{j=r_i}^\infty a_{ij}(z-\zeta^{n_{\sigma^{-1}}}\sigma^{-1}(q) )^j t^i \mapsto \sum \limits_{i=r}^\infty \sum \limits_{j=r_i}^\infty \zeta^{n_\sigma(i-j)}\sigma(a_{ij})(z-q)^j t^i.  \] It is now easy to check that $\del_{t_0} \circ \sigma = \sigma \circ \del_{t_0}$ holds (recall that $\zeta^e=1$). 
\end{proof}

We proceed with a proposition that asserts that if the fundamental solution matrices $Y_P\in \GL_n(F_{\wp(P)})$ for $P \in P_1^\Gamma\cup\dots\cup P_r^\Gamma$ are chosen in a ``$\Gamma$-equivariant'' way, then the $\PPV$-ring $R/K(x)$ obtained in Theorem \ref{thmpatching} descends to a $\PPV$-ring $R_0/K_0(x)$.

\begin{prop}\label{theorem equivariant}
Under the assumptions of Theorem \ref{thmpatching} and in the above setup (i.e., $\mathcal{P}=P_1^\Gamma \cup \dots P_r^\Gamma$), assume that $\sigma(Y_{P^\sigma})=Y_P$ for all $P \in \Pcal$ and all $\sigma \in \Gamma$ (where $\sigma(Y_{P^\sigma}) \in \GL_n(F_{\wp(P)}$ denotes the image of $Y_{P^\sigma}\in \GL_n(F_{\wp(P^\sigma)})$ under $\sigma \colon F_{\wp(P^\sigma)}\to F_{\wp(P)}$ applied to all entries of $Y_{P^\sigma}$). Then we can choose the fundamental solution matrix $Y \in \GL_n(F_U)$ for $A$ obtained in Theorem \ref{thmpatching} such that its entries are $\Gamma$-invariant.
\end{prop}
\begin{proof}
This statement is proved in \cite[Thm. 2.4.b)]{HHM} for the non-parameterized case and since the statement does not depend on $\del_{t_0}$, it also holds in the context of Theorem \ref{thmpatching}. (The strategy of the proof in \cite[Thm. 2.4.b)]{HHM} is to show that for a fixed fundamental solution matrix $Y \in \GL_n(F_U)$, there exists a matrix $B \in \GL_n(F)$ such that the entries of $B\cdot Y$ are $\Gamma$-invariant.)
\end{proof}

 Note that Lemma \ref{lemmainvariantPPVR} (applied to $L=F_U$) assures that the $\PPV$-ring $R=F\{Y,Y^{-1}\}_{\del_{t_0}}$ with $Y$ as in Proposition \ref{theorem equivariant} is of the form $R_0\otimes_{K_0}K$ for a $\PPV$-ring $R_0$ over $K_0(x)$. 

\begin{thm}\label{criterion}
Under the assumptions in Notation \ref{not}, let $\G \leq \GL_n$ be a linear $\del_{t_0}$-algebraic group defined over $k_0((t_0))$. Assume that $\G_{k((t))}=\overline{\left \langle \G_1,\dots,\G_r \right \rangle}^K$ for some Kolchin-closed subgroups $\G_i \leq \G_{k((t))}$ defined over $k((t))$. If for all $i$, and for all finite closed $k$-points $P$ on the $z$-line $\mathbb{P}^1_k$, there exist a $\PPV$-ring $R_P \subseteq F_{\wp(P)}$ with $\PPV$-group $\G_i$ over $F_P$, then there exists a $\PPV$-ring $R$ over $k_0((t_0))(x)$ with $\PPV$-group $\G$. 

(Here we consider $F_P$ and $F_{\wp(P)}$ as $\del\del_{t_0}$-fields via $\del_{t_0}$ as constructed in Lemma \ref{defdelN}. In particular, $\del_{t_0}(t)=t^{1-e}/e$ and $\del_{t_0}(z)=-zt^{-e}/e$). 
\end{thm}
\begin{proof}
Let $\Gamma$ denote the (finite) Galois group of $k((t))/k_0((t_0))$. As explained above, there exist $b_1,\dots,b_r \in k$ such that for the points $P_1,\dots,P_r$ on the $z$-line $\mathbb{P}^1_k$ given by $z=b_1,\dots,z=b_r$ the following holds: the orbits $P_1^\Gamma,\dots,P_r^\Gamma$ (under the action of $\Gamma$ as explained above) are all disjoint and of order $|\Gamma|$. We set $m=r\cdot |\Gamma|$ and consider the set of $m$ points $\Pcal=P_1^\Gamma\cup\dots\cup P_r^\Gamma$. For all $1\leq i \leq r$, there exist $\PPV$-rings $R_{P_i}=F_{P_i}\{Y_{P_i},Y_{P_i}^{-1}\}_{\del_{t_0}}\subseteq F_{\wp(P_i)}$ such that $\GalN_{Y_{P_i}}(R_{P_i}/F_{P_i})=\G_i$, by assumptions. We now define $\PPV$-rings $R_P/F_P$ for an arbitrary $P \in \Pcal$. Let $P \in \Pcal$. Then there exist unique $i \leq r$ and $\sigma \in \Gamma$ with $P^\sigma=P_i$. Recall that we have $\del\del_{t_0}$-differential isomorphisms $\sigma \colon F_{P_i}\to F_P$ and $\sigma \colon F_{\wp(P_i)}\to F_{\wp(P)}$. We set $Y_P=\sigma(Y_{P_i}) \in \GL_n(F_{\wp(P)})$ and $R_P=F_P\{Y_P, Y_P^{-1}\}_{\del_{t_0}}\subseteq F_{\wp(P)}$. This is a $\PPV$-ring over $F_{P}$ for the matrix $\del(Y_{P})Y_{P}^{-1}=\sigma(\del(Y_i)Y_i^{-1}) \in \sigma(F_{P_i}^{n\times n})=F_P^{n\times n}$. We claim that its $\PPV$-group $\G_P:=\GalN_{Y_P}(R_P/F_P)$ is contained in $\G$. Fix a differentially closed field $\hat K \supseteq k((t))$. By Seidenberg's differential Nullstellensatz, it suffices to show that $\G_P(\hat K) \subseteq \G(\hat K)$. Fix an extension of $\sigma$ from $k((t))$ to $\hat K$, i.e. a $\del_{t_0}$-differential isomorphism $\sigma \colon \hat K \to \hat K$ extending $\sigma \colon k((t)) \to k((t))$. 
We obtain a $\del\del_{t_0}$-differential isomorphism $\sigma \colon R_{P_i}\otimes_{k((t))} \hat K \to R_P\otimes_{k((t))} \hat K$.
Hence \[\underline{\Aut}^{\del\del_{t_0}}(R_P/F_P)(\hat{K})= \sigma\underline{\Aut}^{\del\del_{t_0}}(R_{P_i}/F_{P_i})(\hat{K})\sigma^{-1}\] and thus 
 \begin{eqnarray*}
   \G_P(\hat K)=\GalN_{Y_P}(R_P/F_P)(\hat {K})&=&\GalN_{\sigma(Y_{P_i})}(\sigma(R_{P_i})/\sigma(F_{P_i}))(\hat{K})\\
   &=&\sigma(\GalN_{Y_{P_i}}(R_{P_i}/F_{P_i})(\hat{K}))\\
   &=&\sigma(\G_i(\hat {K})) \\
   &\subseteq &\sigma(\G(\hat K))= \G(\hat K),
 \end{eqnarray*}
 since $\G$ is defined over $k_0((t_0))$. 

We can now apply Theorem \ref{thmpatching} to obtain a $\PPV$-ring $R=k((t))(x)\{Y,Y^{-1}\}_{\del_{t_0}}\subseteq F_U$ over $k((t))(x)$ with $\PPV$-group 
\[\GalN_Y(R/F)=\overline{\left \langle \GalN_{Y_P}(R_P/F_P) \mid P \in \Pcal \right \rangle}^K= \overline{\left \langle \G_1,\dots,\G_r \right \rangle}^K=\G_{k((t))}.\] Moreover, we may assume that the entries of $Y \in \GL_n(F_U)$ are $\Gamma$-invariant by Proposition \ref{theorem equivariant}. We can now apply Lemma \ref{lemmainvariantPPVR} (with $L=F_U$) and conclude that $k_0((t_0))(x)\{Y,Y^{-1} \}_{\del_{t_0}}$ is a $\PPV$-ring over $k_0((t_0))(x)$ with $\PPV$-group $\G$. 
\end{proof}

To illustrate how this theorem can be applied, we translate it into a more explicit criterion:
\begin{crit} \label{crit}
Let $\G$ be a linear $\del_t$-linear algebraic group over $k((t))$. Suppose that we would like to show that $\G$ is a $\PPV$-group over $k((t))(x)$, where we consider $k((t))(x)$ as a $\del\del_t$-field via $\del=\del/\del x$ and $\del_t=\del/\del t$. Proceed as follows.
\begin{description}
 \item[(1)] Rename $k_0:=k$, $t_0:=t$ and similarly $\del_{t_0}:=\del_t$.
\item[(2)] Find finitely many $\del_{t_0}$-algebraic subgroups $\G_1,\dots,\G_r$ of $\G_{\overline{k_0((t_0))}}$ such that \\ $\overline{\left \langle \G_1,\dots,\G_r \right \rangle}^K=\G_{\overline{k_0((t_0))}}$.\\ (Choose $\G_1,\dots,\G_r$ such that it seems feasible to construct explicit $\PPV$-rings with $\PPV$-groups $\G_1,\dots,\G_r$. For example, these subgroups should be of small dimensions and of a ``simple structure''. )
\item[(3)] Let $K/k_0((t_0))$ be a finite field extension such that all subgroups $\G_1,\dots,\G_r$ are defined over $K$. 
After enlarging $K$ if necessary, we may assume $K=k((t))$ for a finite Galois extension $k/k_0$ and an $e$-th root $t$ of $t_0$ ($e\geq 1$), see \cite[Lemma 3.4]{HHM}. We may also assume that $k$ contains a primitive $e$-th root of unity $\zeta$. Set $z=x/t$. Note that our notation now conforms to Notation \ref{not}.
\item[(4)] For every $1\leq i \leq r$ and every $q \in k$, construct a $\PPV$-ring $R_{q,i}$ over $k((t,z-q))$ such that $R_{q,i} \subseteq k((z-q))((t))$ and such that the $\PPV$-group of $R_{q,i}/k((t,z-q))$ equals $\G_i$. Here, ``$\PPV$-ring over $k((t,z-q))$'' is meant with respect to the $\del\del_{t_0}$-differential structure on $k((t,z-q))$, where $\del_{t_0}$ is as defined in Lemma \ref{defdelN} and $\del=1/t\cdot \del/\del(z-q)$. 
\end{description}
If Task (4) can be completed succesfully, then Theorem \ref{criterion} asserts that there exists a $\PPV$-ring over $k_0((t_0))(x)$ with $\PPV$-group $\G$. 

\end{crit}

\section{Results}\label{section results}
\begin{lem}\label{building blocks}  
Under the assumptions as in Notation \ref{not}, let $\Hcal\leq \GL_n$ be a linear $\del_{t_0}$-algebraic group defined over $k((t))$ that is $k((t))$-isomorphic as $\del_{t_0}$-algebraic group to either 
\begin{enumerate}
 \item $\Z/r\Z$, where $r \in \N$ is such that $k$ contains a primitive $r$-th root of unity , or
 \item a Kolchin-closed subgroup $H$ of $\Ga$, such that $H=\overline{\left \langle h \right \rangle}^K$ for some $h \in \Ga(k((t)))$, or
 \item the constant subgroup $\Gm^\bigtriangleup$ of $\Gm$.
\end{enumerate}
Then for all closed, finite $k$-points $P$ on the $z$-line $\mathbb{P}^1_k$, there exists a $\PPV$-ring $R_P/F_P$ with $R_P \subseteq F_{\wp(P)}$ and with $\GalN_Y(R_P/F_P)=\Hcal$ for a suitable fundamental solution matrix $Y \in \GL_n(R_P)$. (Here, we consider the fields $F_P$ and $F_\wp(P)$ as $\del\del_{t_0}$-differential fields with $\del$ as defined in Notation \ref{not} and with $\del_{t_0}$ as defined in Lemma \ref{defdelN}.)
\end{lem}
\begin{proof}
Let $P \in \mathbb{P}^1_k$ be a point of the form $z=q$ for some $q \in k$. 

a) It was shown in \cite[Lemma 3.6]{HHM} that $y:=(1-(z-q)^{-1}t)^{1/r}$ is contained in $k((z-q))((t))=F_{\wp(P)}$ and that $y$ is algebraic over $F_P=k((z-q,t))$ of degree $r$, hence $F_P(y)/F_P$ is cyclic of degree $r$. It was furthermore shown in \cite[Prop. 3.7]{HHM} that $R_P:=F_P[y] \subseteq F_{\wp(P)}$ is a (non-parameterized) Picard-Vessiot ring for a matrix $A \in F_P^{n\times n}$ over $F_P$ with fundamental solution matrix $Y \in \GL_n(F_{\wp(P)})$ such that the non-parameterized differential Galois group
 $\underline{\operatorname{Gal}}_Y^\del(R_P/F_P)$ equals $\Hcal$ as subgroup of $\GL_n$ (here we use that $\Hcal \leq \GL_n$ is Zariski-closed, since it is finite). Now $R_P=F_P[y]$ is a finite field extension of $F_P$, hence $\del_{t_0}$ extends uniquely to $R_P$ and we conclude that $R_P$ is also a $\PPV$-ring for $A$ over $F_P$. Fix a differentially closed field $\hat K \supseteq k((t))$. As $R_P/K$ is a regular extension, $R\otimes_K \hat K$ is an integral domain and $\Frac(R\otimes_K \hat K)$ is a finite field extension of $\Frac(F\otimes_K \hat K)$. Hence every automorphism of $\Frac(R\otimes_K \hat K)/\Frac(F\otimes_K \hat K)$ is a $\del\del_{t_0}$-automorphism and thus $\Aut^{\del\del_{t_0}}(R\otimes_K \hat K/F\otimes_K \hat K)=\Aut^\del(R\otimes_K \hat K/F\otimes_K \hat K)$. We conclude that the $\PPV$-group $\GalN_Y(R_P/F_P)$ and the non-parameterized differential Galois group $\underline{\operatorname{Gal}}_Y^\del(R_P/F_P)$ have the same set of $\hat K$-rational points and the claim follows.

b) By Proposition \ref{Tannaka}, we may assume $\Hcal=H$, where we consider $H$ as a subgroup of $\GL_2$ in its representation $\begin{pmatrix}
                                                                                                                               1& * \\ 0& 1
                                                                                                                              \end{pmatrix}$. As explained in Example \ref{ex Kolchin Abschluss}, 
\[H=\{\begin{pmatrix}
             1 & x \\ 0 & 1 
            \end{pmatrix}
 \ | \ h\del_{t_0}(x)-\del_{t_0}(h)x=0 \}=\{\begin{pmatrix}
             1 & x \\ 0 & 1 
            \end{pmatrix}
 \ | \ \del_{t_0}(x/h)=0 \}.\] 

Consider \[f:=\sum \limits_{n=1}^{\infty} \frac{(-1)^{n+1}}{n(z-q)^n}t^n \in k((z-q))((t))=F_{\wp(P)}. \] Recall that $\del(z)=1/t$, $\del_{t_0}(z)=-zt^{-e}/e$, $\del_{t_0}(t)=t^{1-e}/e$ and that we are using the canonical extension of $\del$ from $F$ to $k((z-q))((t))$ and the extension of $\del_{t_0}$ from $F$ to $k((z-q))((t))$ as defined in Lemma \ref{defdelN}. We compute 
\[\del(f)=\frac{1}{t}\sum \limits_{n=1}^{\infty} \frac{(-1)^{n}}{(z-q)^{n+1}}t^n=\frac{-1}{(z-q)^2} \sum \limits_{n=0}^\infty \left(\frac{-t}{z-q}\right)^n=\frac{-1}{(z-q)^2+t(z-q)}, \]
hence $\del(f) \in k((z-q,t))=F_P$. Next, we compute
\[\del_{t_0}(f)=\frac{1}{e}\sum \limits_{n=1}^{\infty} \frac{(-1)^{n+1}}{(z-q)^{n}}t^{n-e}-\frac{z}{e}\sum \limits_{n=1}^{\infty} \frac{(-1)^{n}}{(z-q)^{n+1}}t^{n-e} \] and conclude that also $\del_{t_0}(f)\in  F_P$. Set $y=h\cdot f$ and $Y=\begin{pmatrix}
1 & y \\
0& 1                                                                                                                                                                                                                                             
\end{pmatrix}\in \GL_2(F_{\wp(P)})$. Then $\del(y)=h\del(f) \in F_P$ and $\del_{t_0}(y/h) \in F_P$. This implies that $R_P:=F_P\{Y,Y^{-1}\}_{\del_{t_0}}=F_P\{y\}_{\del_{t_0}}=F_P[f]\subseteq F_{\wp(P)}$ is a $\PPV$-ring over $F_P$ for $A=\begin{pmatrix}
     0 & \del(y)   \\                                                                                                                                                                                                               
     0 & 0                                                                                                                                                                                                                     \end{pmatrix}$. Applying Lemma \ref{AequivalenzUG}.b with $L=h\cdot\del_{t_0}-\del_{t_0}(h)\cdot \del_{t_0}^0$ implies $\GalN_{Y}(R_P/F_P)\leq H$. 
 It is easy to see that $f \notin F_P$ (for a proof, see the proof of Thm. 4.2. in \cite{param_LAG}) and thus $y \notin F_P$, hence $\GalN_{Y}(R_P/F_P)$ is non-trivial by the Galois correspondence (\cite[Proposition 8.5]{Gilletetc}). As $H$ does not have non-trivial subgroups (see Example \ref{ex Kolchin Abschluss}), we conclude $\GalN_{Y}(R_P/F_P)=H$.

c) By Proposition \ref{Tannaka}, we may assume $\Hcal=\Gm^\bigtriangleup$. Consider $y=\exp(\frac{t}{z-q}) \in F_{\wp(P)}=k((z-q))((t))$. Then $\frac{\del(y)}{y}=\del(\frac{t}{z-q}) \in F$, hence $R_P:=F_P\{y,y^{-1}\}_{\del_{t_0}}\subseteq F_{\wp(P)}$ is a $\PPV$-ring over $F_P$. As $\frac{\del_{t_0}(y)}{y}=\del_{t_0}(\frac{t}{z-q}) \in F_P$, $\GalN_y(R_P/F_P)\leq \Gm^{\bigtriangleup}$ by Lemma \ref{AequivalenzUG} (applied with $L=\del_{t_0}^0$). Note that every strict differential algebraic subgroup of $\Gm^{\bigtriangleup}$ is finite. It was shown in the proof of Lemma 3.5 in \cite{HHM} that $y$ is not algebraic over $F_P$. Thus every $c \in \Gm(k)$ defines an $F_P$-linear $\del\del_t$-automorphism of $R_P=F_P[y,y^{-1}]$ with $y \mapsto yc$. Hence ${\Aut}^{\del\del_{t_0}}(R_P/F_P)$ is not finite and thus $\GalN_y(R_P/F_P)$ is not finite and we conclude $\GalN_y(R_P/F_P)= \Gm^{\bigtriangleup}$ .
\end{proof}

\begin{thm}\label{result}
We consider $k((t))$ as a $\del_t$-differential field with $\del_t=\del/\del t$ and $k((t))(x)$ as a $\del\del_t$-field with $\del=\del/\del x$ and $\del_t=\del/ \del t$. Let $\G \leq \GL_n$ be a linear differential algebraic group defined over $k((t))$ and let $\overline{k((t))}$ be an algebraic closure of $k((t))$. Let $\G_1,\dots,\G_r$ be finitely many Kolchin-closed subgroups of $\G$ defined over $\overline{k((t))}$ such that for each $i$, either
\begin{enumerate}
 \item $\G_i$ is finite, or
 \item $\G_i=\overline{\left \langle g \right \rangle}^K$ for some $g \in \GL_n(\overline{k((t))})$ such that $\G_i$ is $\overline{k((t))}$-isomorphic to a Kolchin-closed subgroup of $\Ga$, or
 \item $\G_i$ is $\overline{k((t))}$-isomorphic to the constant subgroup $\Gm^{\bigtriangleup}$ of $\Gm$.
\end{enumerate}
Then if $\G_1,\dots,\G_r$ generate a Kolchin-dense subgroup of $\G_{\overline{k((t))}}$, there exists a $\PPV$-ring $R/k((t))(x)$ with $\PPV$-group $\G$. 
\end{thm}
\begin{proof}
For the sake of consistency of the notation with the notation in the previous section, we rename $k_0:=k$ and $t_0:=t$. First note that for all $i \leq r$ with $\G_i$ finite, we may assume that $\G_i$ is finite and cyclic (by replacing $\G_i$ with a couple of subgroups). We can fix a finite extension $K/k_0((t_0))$ such that all elements $g \in \GL_n(\overline{k_0((t_0))})$ mentioned in b) are contained in $\GL_n(K)$, such that all $\overline{k_0((t_0))}$-isomorphisms mentioned in b) and c) are defined over $K$, and furthermore such that $K$ contains a primitive $|\G_i|$-th root of unity for every $i$ with $\G_i$ finite. After enlarging $K$ if necessary, we may assume $K=k((t))$ for a finite Galois extension $k/k_0$ and an $e$-th root $t$ of $t_0$ ($e\geq 1$), see \cite[Lemma 3.4]{HHM}. We may also assume that $k$ contains a primitive $e$-th root of unity. Thus Notation \ref{not} applies to our situation and the claim follows from Theorem \ref{criterion} together with Lemma \ref{building blocks}.
\end{proof}

As an application of Theorem \ref{result}, we obtain that linear differential algebraic groups that are differentially generated by finitely many unipotent elements are $\PPV$-groups over $k((t))(x)$. 

\begin{thm}\label{result_unipotent}
We consider $k((t))$ as a $\del_t$-differential field with $\del_t=\del/\del t$ and $k((t))(x)$ as a $\del\del_t$-field with $\del=\del/\del x$ and $\del_t=\del/ \del t$.
Let $\G \leq \GL_n$ be a linear differential algebraic group defined over $k((t))$ such that $\G=\overline{\left \langle g_1,\dots,g_r\right \rangle}^K$ for some unipotent elements $g_i \in \GL_n(\overline{k((t))})$. Then there exists a $\PPV$-ring $R/k((t))(x)$ with $\PPV$-group $\G$. 
\end{thm}
\begin{proof}
 It suffices to show that each $\overline{\left \langle g_i \right \rangle}^K$ is of type b) as in Theorem \ref{result}. Let $g \in \GL_n(\overline{k((t))})$ be unipotent. We claim that $\overline{\left \langle g \right \rangle}^K$ is $\overline{k((t))}$-isomorphic to a Kolchin-closed subgroup of $\Ga$. Let $\Hcal \leq \GL_n$ denote the Zariski-closure of $\left \langle g \right \rangle$. Then $\Hcal$ is a commutative, unipotent linear algebraic group defined over $\overline{k((t))}$. Hence $\Hcal$ is  $\overline{k((t))}$-isomorphic as an algebraic group to $\Ga^k$ for some $k\geq 1$ by \cite[Theorem 3.4.7(c)]{Springer}. However, as $\Hcal$ contains a Zariski-dense subgroup generated by one element, we conclude $k\leq 1$. It follows that $\overline{\left \langle g \right \rangle}^K\subseteq \Hcal$ is $\overline{k((t))}$-isomorphic as a differential algebraic group to a Kolchin-closed subgroup of $\Ga$.
\end{proof}

As a corollary, we can generalize the result of \cite{param_LAG} from $k((t))$-split semisimple connected linear algebraic groups to arbitrary semisimple connected linear algebraic groups:

\begin{cor}\label{result_semisimple}
We consider $k((t))$ as a $\del_t$-differential field with $\del_t=\del/\del t$ and $k((t))(x)$ as a $\del\del_t$-field with $\del=\del/\del x$ and $\del_t=\del/ \del t$.
Let $\G \leq \GL_n$ be a semisimple connected linear algebraic group defined over $k((t))$. Then there exists a $\PPV$-ring over $k((t))(x)$ with $\PPV$-group $\G$. 
\end{cor}
\begin{proof}
Let $U_1,\dots,U_m$ be the finitely many root subgroups of $\G$ (defined over $\overline{k((t))}$) and fix $\overline{k((t))}$-isomorphisms of linear algebraic groups 
$u_i\colon \Ga \to U_i$. For $1 \leq i \leq m$, we define unipotent elements of $\G(\overline{k((t))})$  
\[h_i=u_i(1) \text{, } g_i=u_i(t) \text{ and } \tilde g_i=u_i( -t^{-1}).\] By Proposition 3.1 in \cite{param_LAG}, \
\[\G_{\overline{k((t))}}=\overline{\left \langle h_1,g_1,\tilde g_1,\dots,h_m, g_m, \tilde g_m\right \rangle}^K.\] Hence the claim follows from Theorem \ref{result_unipotent}.
\end{proof}

\bibliographystyle{alpha}	
 \bibliography{references_paramII}
\end{document}